\documentclass[12pt]{amsart}%
\usepackage{color}

\usepackage[right=2.7cm,left=2.7cm,top=3.8cm,bottom=3.5cm]{geometry}

\usepackage{float}
\usepackage{amsfonts}
\usepackage[pdftex]{graphicx}
\usepackage{amsmath}
\usepackage{amssymb}
\usepackage{flafter}
\usepackage[pdffitwindow=false, pdfview=FitH, pdfstartview=FitH,
plainpages=false]{hyperref}%
\providecommand{\U}[1]{\protect \rule{.1in}{.1in}}
\newtheorem{theorem}{Theorem}[section]
\newtheorem{lemma}[theorem]{Lemma}

\theoremstyle{definition}

\newtheorem{example}[theorem]{Example}

\theoremstyle{remark}

\numberwithin{equation}{section}
\newcommand{\iii}{{\, \vert\kern-0.25ex\vert\kern-0.25ex\vert\, }}

\begin{document}
\thispagestyle{empty}
\title[Petrovsky-Wave Nonlinear coupled system with strong damping]{Uniform Stabilization of the Petrovsky-Wave Nonlinear coupled system with strong damping}
\author[A. Ben Aissa]{Akram Ben Aissa*}

\date{}
\begin{abstract}This paper concerns the well-posedness and uniform stabilization of  the Petrovsky-Wave Nonlinear coupled system with strong damping. Existence of global weak solutions for this problem is established by using the Galerkin method. Meanwhile, under a clever use of the multiplier method, we estimate the total energy decay rate.
\end{abstract}

\keywords{Coupled systems, Nonlinear strong  damping, Well-posedness, Faedo-Galerkin,  General decay, Multiplier method, Convexity.\\
*:UR Analysis  and Control  of PDE's, UR 13ES64\\
Higher Institute of transport and Logistics of Sousse, University of Sousse,  Tunisia.\\
email:akram.benaissa@fsm.rnu.tn}

\maketitle

\section{Introduction}

For simplicity reasons, we omit the space variable $x$ of $u(x, t), u_t(x, t)$ and we  denote $u(x, t) = u, u_t(x, t) = u'$ and $u_{tt}(x,t)=u''$. In addition,  when no confusion arises, the functions considered are all real valued.\\
 Our main interest lies in the following system of the coupled  Petrovsky-wave system  of the type
\small{\begin{equation}\label{P}
\left\{
\begin{aligned}
& u''_1+ \Delta^2 u_1 - a(x) \Delta u_2- g_1(\Delta  u'_1) =0,
 \  &x\in \Omega, t \geq 0& \\
 &u''_2- \Delta u_2 - a(x) \Delta u_1- g_2( \Delta u'_2)=0 \ &x\in \Omega, t \geq 0&
 \\&\Delta u_1 =u_1=u_2 = 0, \  &x\in \Gamma, t \geq 0& \\
&u_i(x,0)=u_{i}^0(x), \, \,\, u'_{i}(x,0)=u_{i}^1(x),\   &x\in \Omega,\,\,\,\,\, i=1,2.& \\
\end{aligned}
\right.
\end{equation}}
Here $ \Omega $ is a bounded domain of $ \mathbb{R}^n $ with regular boundary $ \Gamma $.\\
When $a(x)=0$, the  Petrovsky equation was treated by Komornik  {\bf\cite{KOM}}, 
%studied the following nonlinear Petrovsky system  with a strong damping
%paper, the author has considered the equation ,
%\small{\begin{equation*}
%\left\{
%\begin{aligned}
%& u''(x, t)+\Delta^2 u(x, t)-g(\Delta u')=0, \ &x\in \Omega \times [0,+\infty[&,
%\\&u(0, t)=\Delta u=0, \  &x\in \Gamma \times [0,\infty[&,\\
%&u(x,0)=u_{0}(x), \, ,\,\, u_{t}(x,0)=u_{1}(x) &x\in \Omega \times [0,+\infty[,& \\
%\end{aligned}
%\right.
%\end{equation*}}
where he used  semigroup approach for setting the well possedness and
% integral inequalities for the stability of a weak solutions.\\
he studied the  strong  stability 
%of this system  
by introducing a  multiplier method combined with a nonlinear integral inequalities.
Recently,  Bahlil et al. \cite{Feng}, studied the  system 
\begin{equation}\label{K1}
\left\{
\begin{array}{lll}
u_{1}''+a(x) u_2+\Delta^{2}u_1-g_1(
u_{1}'(x,t))=f_1(u_1, u_2), & in & \Omega \times \mathbb{R}^{+}, \\
u_{2}''+a(x) u_1-\Delta u_2-g_2(
u_{2 }'(x,t))=f_2(u_1, u_2), & in & \Omega \times \mathbb{R}^{+}, \\
\partial_{\nu}u_1= u_1=v= u_2 =0 & on & \Gamma \times \mathbb{R}^{+}, 
\end{array}%
\right.        
\end{equation}
for $g_i\,\, (i=1,2)$  do not necessarily having a polynomial growth near the origin, by using Faedo-Galerkin method to prove the existence and uniqueness of solution and established energy decay results depending on $g_i$. 
%A. Guesmia {\bf\cite{Ga}} studied the system (\ref{K1}) with a nonlinear damping $g(u_i')$. 
%He used  semigroup approach for sitting the well possedness and
% integral inequalities for the stability of a weak solutions.\\
%he showed the uniform exponential  and polynomial decay of solution  by introducing a  multiplier method combined with a nonlinear integral
%inequalities given by Martinez {\bf\cite{Ma}}.\\\\
Guesmia {\bf\cite{Ga}}  consider the problem (\ref{K1}) without Source
Terms $f_1$ and $f_2$. He  deal with global existence and uniform decay of solutions.\\

In this paper, we prove the global existence of weak solutions of the problem (\ref{P}) by
using the Galerkin method (see Lions  {\bf\cite{Lio}}) 
we use some technique from {\bf\cite{Feng}} to establish an explicit and general decay result, depending on $g_i$. The proof is based on a powerful tool which is  the multiplier method \cite{lions,KOM1} and makes use of some properties of convex functions, and  general  Jensen and Young's inequalities. These convexity arguments were introduced
and developed by Lasiecka and co-workers ({\bf\cite{Las}},{\bf\cite{Las1}}) and exploited later on, with appropriate modifications, by Liu and Zuazua {\bf\cite{Liu}}, Alabau-Boussouira {\bf\cite{Ala}} and others.

The paper is organized as follows. In section 2  we present some assumptions and  technical lemmas. In section 3 we prove the existence and the uniqueness  of a global solution.  In section 4 we prove the energy estimates.
\section{Functional setting and statement of main results }

Let us introduce for brevity the following Hilbert spaces
$$
H=L^2(\Omega) \times L^2(\Omega)
$$
$$
W=H_0^1(\Omega)\times H_0^1(\Omega)
$$
$$
H_\Delta^3(\Omega)=\{ v\in H^3(\Omega) | v=\Delta v=0 \hbox{ on } \Gamma\},\ \ \  \|v\|^2_{H_\Delta^3(\Omega)}=\int_{\Omega}|\nabla\Delta v|^2dx
$$
$$
V=H_\Delta^3(\Omega)\cap H^2(\Omega) \times H^2(\Omega)
$$
$$
\widetilde{V}=(H^4(\Omega) \cap H_\Delta^3(\Omega))\times  (H_\Delta^3(\Omega) \cap H^2(\Omega)).
$$
Identifying $H$ with its dual, we obtain the diagram
$$
\widetilde{V} \subset V\subset W \subset  H=H'\subset W' \subset V' \subset \widetilde{V}'.
$$
We impose the following assumptions on $a$ and $g_i$\\
%Let $\alpha$ be a positive constant such that
$\blacktriangleright$  The function $a:\Omega \rightarrow  \mathbb{R} $ is a nonnegative and bounded such that
 \begin{equation}\label{F1d}
 \begin{split}
& a(x)\in W^{1 , \infty } (\Omega).
\\ & \|a\|_{L^\infty(\Omega)}  <\min\Big \{\frac{1}{c'},1\Big\}
  \end{split}
  \end{equation}
  where $c'> 0$ (depending only on the geometry of $\Omega$) is  the constant
%The function $a$ belongs to $L^\infty(\Omega)$
%\begin{equation}\label{F1d}
%\|a\|_{L^\infty(\Omega)}  < \frac{1}{c'}
%\end{equation}
%where $c'> 0$ (depending only on the geometry of $\Omega$) is  the constant
$$
\|\Delta  v\| \leq c'\|\nabla \Delta v\|,\,\,\,\,\,  \forall v\in H^3_\Delta(\Omega).
$$
$$
\|\nabla v\| \leq c\| \Delta v\|,\,\,\,\,\,  \forall v\in H^2_0(\Omega).
$$
%We impose the following assumptions on $g_1$ and $g_2$\\
%To state and prove our result, we use the following assumptions:
%\begin{itemize}
%\item[(A)] 
 $\blacktriangleright\;\;g_i:\mathbb{R}\to \mathbb{R}$ be non decreasing convex function of class $\mathcal{C}^1$ such that there exists 
$\epsilon$ (sufficiently small), $c_i,\,\, \tau_i> 0,\,\, (i=1,2)$, and
 $G:{\mathbb{R}_+}\to {\mathbb{R}_+}$ is convex, increasing and of  class 
$\mathcal{C}^1(\mathbb{R}_+)\cap \mathcal{C}^2(]0,+\infty[)$ satisfying
 \begin{equation}\label{g1}
 \begin{gathered}
\text{$G(0)=0$ and $G$ is linear on $[0,\epsilon]$ or}\\
\text{$G'(0)=0$ and $G''>0 $  on $]0,\epsilon]$ such that}\\
c_1|s|\leq  |g_i(s)|\leq c_2|s|\quad \text{if } |s|>\epsilon \\
s^2+ g_i^2(s)\leq G^{-1}(sg_i(s))\quad \text{if } |s|\leq \epsilon,\\
\exists  \tau_1,\, \tau_2 >0,\,\, \tau_1 \leq  g_i'(s) \leq \tau_2,\,\,\,  \forall s \in  \mathbb{R}.
 \end{gathered}
\end{equation}
We are now in a position to state our  main  results.
\begin{theorem}\label{H}
Let $(u_1^0,u_2^0)\in \widetilde{V}$ and $(u_1^1,u_2^1)\in V$ arbitrarily. Assume that (\ref{F1d}) and (\ref{g1}) hold. Then, system (\ref{P})
has a unique weak solution satisfying
$$
(u_1,u_2)\in L^{\infty}(\mathbb{R}_+,\widetilde{V}),\,\,\,\,\,(u'_1,u'_2)\in L^{\infty}(\mathbb{R}_+,V)
$$
and 
$$
(u''_1,u''_2)\in L^{\infty}(\mathbb{R}_+,W)
$$
\end{theorem}
%\begin{theorem}(Regular solutions)\label{HH}
%then System (\ref{P}) has a solution $(u_1,u_2)$ satisfying
%$$
%(u_1,u_2)\in \mathcal{C}(\mathbb{R}_+,\widetilde{V}) \cap \mathcal{C}^1(\mathbb{R}_+,V)
%$$
%\end{theorem}
\begin{theorem}
\label{G}
Let $(u^0_1, u^0_2)  \in \widetilde{V} $ and 
$(u^1_1, u_2^1)  \in  V $. Assume that (\ref{F1d}) and  (\ref{g1}) hold.  The energy of the unique solution of system \eqref{P}, given by  (\ref{E1}) decays as
%then the global solutions of the problem \eqref{P} have the following asymptotic
%property
\begin{equation}\label{77}
E(t)\leq \psi^{-1}\Big(h(t)+\psi(E(0))\Big),\,\, \forall  t\geq 0
\end{equation}
where $\psi(t)=\displaystyle\int_t^1 \frac{1}{\omega \varphi(s)}\,ds$ for $t>0$,\,\,\, $h(t)=0$ for $0\leq t\leq \frac{E(0)}{\omega \varphi(E(0))}$ and
$$
h^{-1}(t)=t+\frac{\psi^{-1}(t+\psi(E(0)))}{\varphi(\psi^{-1}(t+\psi(E(0))))} ,\,\, \forall  t\geq \frac{E(0)}{\varphi(E(0))}
$$
\begin{gather}\label{W1}
%H_1(t)=\int_t^1\frac{1}{H_2(s)}\,ds,  \\
%\label{p}
\varphi(t)=\begin{cases}
t & \text{if $G$ is linear on } [0,\varepsilon]\\
 tG'(\varepsilon_0t) & \text{if $G'(0)=0$ and $G''>0$ on } ]0,\varepsilon],
 \end{cases} \nonumber
\end{gather}
where  $\omega$  and $\varepsilon_0$  are positive constants.
\end{theorem}

\begin{lemma}\label{L1A}
The energy functional associated to the solution of the problem (\ref{P}) given by the following formula%
\begin{equation}\label{E1}
E(t)=\frac{1}{2} \int_\Omega |\nabla u'_{1}|^2+|\nabla u'_{2}|^2+|
\nabla \Delta u_{1}|^2+|\Delta u_{2}|^{2}\,dx+ \int_{\Omega } a(x)\Delta u_{1}\Delta u_{2}dx,
\end{equation}
is a nonnegative.
\end{lemma}
\begin{proof}
Multiplying the first equation in (\ref{P}) by $-\Delta u'_1$ and the second equation by $-\Delta u'_2$, integrating (by parts) over $\Omega$, we obtain
\begin{equation*}
\begin{split}
\frac{1}{2}\frac{d}{dt} &\Big[\int_\Omega |\nabla u'_{1}|^2+|\nabla u'_{2}|^2+|
\nabla \Delta u_{1}|^2+|\Delta u_{2}|^{2}\,dx+ 2\int_{\Omega } a(x)\Delta u_{1}\Delta u_{2}\,dx\Big]\\&
=-\int_\Omega \Delta u'_1g_1(\Delta u'_1)+\Delta u'_2g_2(\Delta u'_2)\,dx.
\end{split}
\end{equation*}
Using H\"{o}lder's inequality, Sobolev embedding and the condition (\ref{F1d}), we get
\begin{equation*}
\begin{split}
 \int_{\Omega } a(x)\Delta u_{1}\Delta u_{2}dx &\geq -\frac{1}{2}\|a\|_{L^\infty(\Omega)} \frac{\sqrt{c'}}{\sqrt{c'}} \int_\Omega  |\Delta u_{1}\Delta u_{2}|\,dx
 \\&\geq -\frac{1}{2}\|a\|_{L^\infty(\Omega)} \int_\Omega  \frac{1}{c'}|\Delta u_{1}|^2 +c'|\Delta u_{2}|^2\,dx
  \\&\geq -\frac{1}{2}\|a\|_{L^\infty(\Omega)} \int_\Omega  \frac{c'^2}{c'}|\nabla \Delta u_{1}|^2 +c'|\Delta u_{2}|^2\,dx
\\& \geq-\frac{c'}{2}\|a\|_{L^\infty(\Omega)} \int_\Omega   |\nabla\Delta  u_{1}|^2 +|\Delta u_{2}|^2\,dx
\end{split}
\end{equation*}
then
\begin{equation*}
\begin{split}
E(t)&\geq \frac{1}{2} \int_\Omega |\nabla u'_{1}|^2+|\nabla u'_{2}|^2+(1-c'\|a\|_{L^\infty(\Omega)})(|
\nabla \Delta u_{1}|^2+|\Delta u_{2}|^{2})\,dx\\&\geq 0.
\end{split}
\end{equation*}
Hence, $E$ is a nonnegative function and its derivative is 
\begin{equation}\label{F44}
E'(t)=-\int_\Omega \Delta u'_1g_1(\Delta u'_1)+\Delta u'_2g_2(\Delta u'_2)\,dx.
\end{equation}
\end{proof}

%%%%%%%%%%%%%%%
\section{Some technical lemmas}
\begin{lemma}\label{0}
Let $E: \mathbb{R}^+ \rightarrow \mathbb{R}^+$  be a non-increasing differentiable
function, $\lambda \in \mathbb{R}^+ $ and $\varphi: \mathbb{R}^+ \rightarrow \mathbb{R}^+$ a convex and increasing function such that $\varphi(0)=0$. Assume that
\begin{equation*}
\left\{
\begin{aligned}
& \int_s ^{+\infty}\varphi(E(t))\,dt \leq  E(s),
 \  & \forall  s\geq 0& \\
 &E'(t)\leq \lambda E(t) \ & \forall  t\geq 0.&\\
 %\\&\Delta u_1 =u_1=u_2 = 0, \  &x\in \Gamma, t \geq 0& \\
%&u_i(x,0)=u_{i}^0(x), \, \,\, u'_{i}(x,0)=u_{i}^1(x),\   &x\in \Omega,\,\,\,\,\, i=1,2& \\
\end{aligned}
\right.
\end{equation*}
Then $E$ satisfies the following estimate:
\begin{equation}\label{70K}
E(t)\leq e^{\tau_0 \lambda} d^{-1}\Big(e^{\lambda(t-h(t))}\varphi\Big(\psi^{-1}\Big(h(t)+\psi(E(0))\Big)\Big),\,\, \forall  t\geq 0
\end{equation}
where $$
\psi(t)=\int_t^1 \frac{1}{\varphi(s)}\,ds,\,\,\,\,\,\,\,  \forall  t\geq 0
$$ 
\begin{equation*}
d(t)=\left\{
\begin{aligned}
& \Psi(t),
  \  & \mbox{if}\,\,\, \lambda=0&  \\
 &\int_0^t\frac{\varphi(s)}{s}\,ds \ & \mbox{if}\,\,\,  \lambda >0&\\
 %\\&\Delta u_1 =u_1=u_2 = 0, \  &x\in \Gamma, t \geq 0& \\
%&u_i(x,0)=u_{i}^0(x), \, \,\, u'_{i}(x,0)=u_{i}^1(x),\   &x\in \Omega,\,\,\,\,\, i=1,2& \\
\end{aligned}
\right.
\end{equation*}
\begin{equation*}
h(t)=\left\{
\begin{aligned}
& K^{-1}(D(t)),
  \  & \forall  t>T_0& \\
 &0  \  & \forall  t \in [0,T_0]& \\
 %\\&\Delta u_1 =u_1=u_2 = 0, \  &x\in \Gamma, t \geq 0& \\
%&u_i(x,0)=u_{i}^0(x), \, \,\, u'_{i}(x,0)=u_{i}^1(x),\   &x\in \Omega,\,\,\,\,\, i=1,2& \\
\end{aligned}
\right.
\end{equation*}
%\,\,\, $h(t)=0$ for $0\leq t\leq \frac{E(0)}{\Psi(E(0))}$ and
$$
K(t)=D(t)+\frac{\psi^{-1}(t+\psi(E(0)))}{\varphi(\psi^{-1}(t+\psi(E(0))))} e^{\lambda t},\,\, \forall  t\geq 0
$$
$$
D(t)=\int_0^t e^{\lambda s}\,ds\,\,\,\,\,\,\,  \forall  t\geq 0
$$
$$
T_0=D^{-1}\Big(\frac{E(0)}{\varphi(E(0))}\Big),\,\,\,\,\, \tau_0=\left\{
\begin{aligned}
& 0 
 \  & \forall  t >T_0& \\
 &T_0 \ & \forall  t \in [0,T_0]&\\
 %\\&\Delta u_1 =u_1=u_2 = 0, \  &x\in \Gamma, t \geq 0& \\
%&u_i(x,0)=u_{i}^0(x), \, \,\, u'_{i}(x,0)=u_{i}^1(x),\   &x\in \Omega,\,\,\,\,\, i=1,2& \\
\end{aligned}
\right.
$$
\end{lemma}
\section{Proof of  Theorem \ref{H} }
We will use the Faedo-Galerkin method \cite{Lio} to prove the existence of a global solutions. Let $T > 0$  be fixed and
denote by $V^k$ the space generated by $\{w^1_i, w^2_i,... ,w^k_i\}$, where the set $\{w_i^k,\,\, k\in \mathbb{N}\} $ is a
basis of $\widetilde{V}$.\\
We construct approximate solution $u_i^k,\,\,  k = 1, 2, 3,.....$ in the form
$$ 
u_i^k(x,t)=\sum_{j=1}^k c^{j k }(t)w_i^j(x),
$$
 where $c^{jk}\,\, (j = 1, 2, . . .,k)$ are determined by the following ordinary differential
equations
\small{\begin{equation}\label{PO}
\left\{
\begin{aligned}
&  (\ddot{u}^k_1+ \Delta^2 u^k_1-a(x) \Delta u^k_2-g_1(\Delta  \dot{u}^k_1) ,w_1^j)=0
 \  &\forall  w_j^1\in V^k& \\
 &(\ddot{u}^k_2-\Delta u^k_2-a(x) \Delta  u^k_1-g_2(\Delta  \dot{u}^k_2) ,w_2^j)=0 \ &\forall w_j^2 \in V^k&
 \\
 %&\Delta u_1 =u_1=u_2 = 0, \  &x\in \Gamma, t \geq 0& \\
&u^k_i(0)=u_{i}^{0k}, \, \,\, \dot{u}_{i}^k(0)=u_{i}^{1k},\   &x\in \Omega,\,\,\,\,\, i=1,2& \\
\end{aligned}
\right.
\end{equation}}
 with initial conditions
\begin{equation}\label{p7}
u^k_1(0)=u^{0k}_1=\sum_{j=1}^k\langle u_1^0,w^j_1\rangle w^j_1\to u_1^0,\quad \text{in }
H^4(\Omega) \cap H^3_\Delta (\Omega) \  \text{ as } k\to +\infty,
\end{equation}
\begin{equation}\label{p72}
u^k_2(0)=u^{0k}_2=\sum_{j=1}^k\langle u_2^0,w^j_2\rangle w^j_2\to u_2^0,\quad \text{in }
H^3_\Delta (\Omega) \cap H^2(\Omega)\text{ as } k\to +\infty,
\end{equation}
\begin{equation}\label{p8}
\dot{u}^k_1(0)=u^{1k}_1=\sum_{j=1}^k\langle u^1_1,w^j_1\rangle w^j_1\to u_1^1,\quad \text{in }
H^3_\Delta (\Omega) \cap  H^2(\Omega) \text{ as } k\to +\infty.
\end{equation}
\begin{equation}\label{p82}
\dot{u}^k_2(0)=u^{1k}_2=\sum_{j=1}^k\langle u^1_2,w^j_2\rangle w^j_2\to u_2^1,\quad \text{in }
H^2(\Omega)\text{ as } k\to +\infty.
\end{equation}
\begin{equation}\label{p8T}
-\Delta^2u^{0k}_1 +a(x) \Delta u^{0k}_2+g_1(\Delta u^{1k}_1)\longrightarrow -\Delta^2u^0_1 +a(x)\Delta u^0_2+g_1(\Delta u^1_1),\quad \text{in }
H^{1}_{0}(\Omega)\text{ as } k\to +\infty.
\end{equation}
\begin{equation}\label{p8F}
\Delta u^{0k}_2+a(x) \Delta u^{0k}_1 +g_2(\Delta u^{1k}_2)\longrightarrow \Delta u^0_2+a(x)\Delta u^0_1 +g_2(\Delta u^1_2),\quad \text{in }
H^{1}_{0}(\Omega)\text{ as } k\to +\infty.
\end{equation}
First, we are going to use some a priori estimates to show that $ t_k = \infty$. Then, we will show that the
sequence of solutions to (\ref{PO}) converges to a solution of (\ref{P}) with the claimed smoothness.\\
Choosing  $w^j_i=-2\Delta \dot{u}_i^k$  in (\ref{PO}), we obtain
\begin{equation}\label{EA0}
\begin{split}
&\frac{d}{dt}\int_\Omega |\nabla \dot{u}^k_{1}|^2+|\nabla \dot{u}^k_{2}|^2+|
\nabla \Delta u^k_{1}|^2+|\Delta u^k_{2}|^{2}\,dx+2 a(x)\Delta  u^k_{1}\Delta u^k_{2}\,dx\\&+2\int_\Omega \Delta \dot{u}^k_1 g_1(\Delta \dot{u}^k_1)\,dx
 +2\int_\Omega \Delta \dot{u}^k_2 g_2(\Delta \dot{u}^k_2)\,dx=0,
 \end{split}
\end{equation}
and choosing  $w^j_i=\Delta^2 \dot{u}_i^k$  in (\ref{PO}), implies
\begin{equation}\label{M1}
\begin{split}
&\frac{d}{dt}\int_\Omega |\Delta \dot{u}^{k}_{1}|^2+|\Delta \dot{u}^{k}_{2} |^2+|\Delta^2 u^{k}_{1} |^2+|\nabla \Delta u^{k}_{2} |^2
+2a(x)\nabla  \Delta u_1^k\nabla \Delta u_2^k\,dx\\&
+2\int_\Omega \nabla a(x)\Delta u_2^k\nabla \Delta \dot {u}_1^k\,dx+2\int_\Omega \nabla a(x)\Delta u_1^k\nabla \Delta \dot {u}_2^k\,dx
\\&
+2\int_\Omega  |\nabla \Delta \dot{u}^{k}_{1}|^2 g'_1(\Delta \dot{u}^k_1)\,dx+2\int_\Omega  |\nabla \Delta \dot{u}^{k}_{2} |^2 g_2'(\Delta \dot{u}^k_2)\,dx=0
\end{split}
\end{equation}
Summing (\ref{EA0}) and (\ref{M1}), we obtain
\begin{equation}\label{MM1}
\begin{split}
&\frac{d}{dt}\int_\Omega \{|\Delta \dot{u}^{k}_{1} |^2+|\Delta \dot{u}^{k}_{2}|^2+ |\nabla \dot{u}^k_{1}|^2+|\nabla \dot{u}^k_{2}|^2+|\Delta^2 u^{k}_{1} |^2+|\nabla \Delta u^{k}_{2} |^2|+
\nabla \Delta u^k_{1}|^2+|\Delta u^k_{2}|^{2}\}\,dx
\\&+2\frac{d}{dt}\int_\Omega \{a(x)\Delta  u^k_{1}\Delta u^k_{2}+a(x)\nabla  \Delta u_1^k\nabla \Delta u_2^k\}\,dx+2\int_\Omega \Delta \dot{u}^k_1 g_1(\Delta \dot{u}^k_1)\,dx
 +2\int_\Omega \Delta \dot{u}^k_2 g_2(\Delta \dot{u}^k_2)\,dx\\&
+2\int_\Omega \nabla a(x)\Delta u_2^k\nabla \Delta \dot {u}_1^k\,dx+2\int_\Omega \nabla a(x)\Delta u_1^k\nabla \Delta \dot {u}_2^k\,dx
\\&
+2\int_\Omega  |\nabla \Delta \dot{u}^{k}_{1}|^2 g_1'(\Delta \dot{u}^k_1)\,dx+2\int_\Omega  |\nabla \Delta \dot{u}^{k}_{2}|^2 g_2'(\Delta \dot{u}^k_2)\,dx=0.
\end{split}
\end{equation}
Using H\"{o}lder's inequality and Sobolev embedding, we have
\begin{equation}\label{V1}
\begin{split}
&2\Big|\int_\Omega  a(x)\Delta u_2^k\Delta u_1^k\,dx\Big|\leq 2\frac{\sqrt{c'}}{\sqrt{c'}}\int_\Omega |a(x)||\Delta u_2^k|| \Delta u_1^k|\,dx
\\&\leq  c'\|a\|  \int_\Omega  |\nabla \Delta u^{k}_{1}(x,t) |^2 \,dx+ c'\|a\| \int_\Omega |\Delta u_2^k(x,t)|^2\,dx
%\\&\leq 
%\frac{1}{2}\int_\Omega  |\nabla \Delta \dot{u}^{k}_{1}(x,t) |^2 g_1'(\Delta \dot{u}^k_1(x,t))\,dx+\frac{c'}{2\tau_1}\|\nabla a\|^2\int_\Omega |\nabla \Delta u_2^k(x,t)|^2\,dx
\end{split}
\end{equation}
and
\begin{equation}\label{M4}
\begin{split}
&\Big| 2 \int_\Omega a(x) \nabla \Delta u_1^k\nabla \Delta u_2^k\,dx\Big|\\&\leq 2\|a\|\int_\Omega  |\nabla \Delta u_1^k||\nabla \Delta u_2^k|\,dx
\\&\leq 
\|a\| \int_\Omega  |\nabla \Delta u^{k}_{1}|^2 \,dx+\|a\|\int_\Omega |\nabla \Delta u_2^k|^2\,dx.
\end{split}
\end{equation}
By  H\"{o}lder's inequality, Sobolev embedding and the condition (\ref{g1}), we get
%By Holder inequality and (\ref{g1}), we have
\begin{equation}\label{M2}
\begin{split}
&2\Big|\int_\Omega \nabla a(x)\Delta u_2^k\nabla \Delta \dot {u}_1^k\,dx\Big|\leq 2\int_\Omega |\nabla a(x)||\Delta u_2^k||\nabla \Delta \dot {u}_1^k|\,dx\\&
\leq  2\int_\Omega |\nabla a(x)||\Delta u_2^k||\nabla \Delta \dot {u}_1^k|\frac{\sqrt{g_1'(\Delta \dot{u}^k_1)}}{\sqrt{\tau_1}}\,dx\\&\leq 
\int_\Omega  |\nabla \Delta \dot{u}^{k}_{1} |^2 g_1'(\Delta \dot{u}^k_1)\,dx+\frac{1}{\tau_1}\|\nabla a\|^2\int_\Omega |\Delta u_2^k|^2\,dx
%\\&\leq 
%\frac{1}{2}\int_\Omega  |\nabla \Delta \dot{u}^{k}_{1}(x,t) |^2 g_1'(\Delta \dot{u}^k_1(x,t))\,dx+\frac{c'}{2\tau_1}\|\nabla a\|^2\int_\Omega |\nabla \Delta u_2^k(x,t)|^2\,dx
\end{split}
\end{equation}
Similarly, we have
\begin{equation}\label{M3}
\begin{split}
2\Big|\int_\Omega \nabla a(x)\Delta u_1^k\nabla \Delta \dot {u}_2^k\,dx\Big|
&\leq 
\int_\Omega  |\nabla \Delta \dot{u}^{k}_{2} |^2 g_2'(\Delta \dot{u}^k_2)\,dx+\frac{1}{\tau_1}\|\nabla a\|^2\int_\Omega |\Delta u_1^k|^2\,dx
\\&\leq 
\int_\Omega  |\nabla \Delta \dot{u}^{k}_{2}|^2 g_2'(\Delta \dot{u}^k_2)\,dx+\frac{c'}{\tau_1}\|\nabla a\|^2\int_\Omega |\nabla \Delta u_1^k|^2\,dx
\end{split}
\end{equation}
Reporting (\ref{V1})-(\ref{M3}), into (\ref{MM1}) and integrating over $(0,t)$,  we find
%Integrating the last inequality over (0, t) and using GronwallÕs lemma, we have
%so that, thanks to the monotonicity condition on the function $g$, one easily derives
\begin{equation*}
\begin{split}
&F^k(t)+2\int_0^t \int_\Omega \Delta \dot{u}^k_1(s)g_1(\Delta \dot{u}^k_1(s))\,dx\,dt
 +2\int_0^t\int_\Omega \Delta \dot{u}^k_2(s) g_2(\Delta \dot{u}^k_2(s))\,dx\,dt
\\&+\int_0^t\int_\Omega  |\nabla \Delta \dot{u}^{k}_{1}(s) |^2 g'_1(\Delta \dot{u}^k_1(s))\,dx\,dt+
\int_0^t\int_\Omega  |\nabla \Delta \dot{u}^{k}_{2}(s) |^2 g'_2(\Delta \dot{u}^k_2(s))\,dx\,dt
%\\&\leq F^k(0)-
%2\alpha \int_\Omega \Delta u_1^m(t)\Delta u_2^m(t)\,dx+2\alpha \int_\Omega \Delta u_1^{0m}\Delta u_2^{0m}\,dx
\\&\leq  F^k(0)+C_1\int_0^t F^k(s)\,dx\,ds
,\,\,\,\,\,\, \forall t\in [0,t_k)
\end{split}
\end{equation*}
where 
\begin{equation*}
\begin{split}
F^k(t)&=\int_\Omega |\Delta \dot{u}^{k}_{1}(t)|^2+|\Delta \dot{u}^{k}_{2}(t) |^2+|\nabla \dot{u}^{k}_{1}(t)|^2+|\nabla \dot{u}^{k}_{2}(t)|^2+|\Delta^2 u^{k}_{1}(t)|^2\,dx\\&+(1-c'\|a\|-|a\|)\int_\Omega|\nabla \Delta u^{k}_{1}(t) |^2\,dx+(1-c'\|a\|)\int_\Omega  |\Delta u^{k}_{2}(t) |^2\,dx+(1-\|a\|)\int_\Omega  |\nabla \Delta u^{k}_{2}(t) |^2\,dx
\end{split}
\end{equation*}
and $C_1$  is a positive constant depending only on $\|a\|,\,\, \|\nabla a\|$ and $\tau_1$.\\
So that, thanks to the monotonicity condition on the function $g_i$ and using Gronwall's lemma, we
conclude that
\begin{equation}\label{S1}
 u^k_1\,\,\,\,\, \hbox{is bounded in } \,\,\,\,\,\,\, L^\infty(0,T;H^4(\Omega)\cap H^3_\Delta (\Omega))
\end{equation}
\begin{equation}\label{S2}
 u^k_2\,\,\,\,\, \hbox{is bounded in } \,\,\,\,\,\,\, L^\infty(0,T;H^3_\Delta (\Omega)\cap H^2 (\Omega))
\end{equation}
\begin{equation}\label{S3}
 \dot{u}^k_{1}\,\,\,\,\, \hbox{is bounded in } \,\,\,\,\,\,\, L^\infty(0,T;H^2(\Omega)\cap H^1_0(\Omega))
 \end{equation}
 \begin{equation}\label{S4}
 \dot{u}^k_{2}\,\,\,\,\, \hbox{is bounded in } \,\,\,\,\,\,\, L^\infty(0,T;H^2(\Omega)\cap H^1_0 (\Omega))
 \end{equation}
 \begin{equation}\label{S44}
 \Delta  \dot{u}^k_i g_i(\Delta \dot{u}^k_i) \,\,\,\,\, \hbox{is bounded in } \,\,\,\,\,\,\,L^1(\mathcal{A}).
%\end{gather}
%where $\mathcal{A}=\Omega\times(0,T)$.
 \end{equation}
 where $\mathcal{A}=\Omega\times(0,T)$.\\
We assume first $t<T$ and let $0<\xi <T-t$.
Set 
$$u_i^{k\xi}(x, t) =u_i^k(x,t+\xi),$$ 
$$U^{k\xi}=u_1^k(x,t+\xi)-u_1^k(x,t),$$
and 
$$D^{k\xi}=u_2^k(x,t+\xi)-u_2^k(x,t).$$
 Then, $U^{k\xi}$ solves the differential equation
\begin{equation}\label{L1}
(\ddot{U}^{k\xi}+\Delta^2U^{k\xi}-a(x)\Delta D^{k\xi} -(g_1(\Delta \dot{u}^{k\xi}_1)-g_1(\Delta \dot{u}^{k}_1)) ,w^j_1)=0,\,\,\,\,\,  \forall w^j_1\in V^k.
\end{equation}
and $D^{k\xi}$ solves
\begin{equation}\label{L2}
(\ddot{D}^{k\xi}-\Delta D^{k\xi}-a(x)\Delta  U^{k\xi} -(g_2(\Delta \dot{u}^{k\xi}_2)-g_2(\Delta \dot{u}^{k}_2)) ,w^j_2)=0,\,\,\,\,\,  \forall w^j_2\in V^k.
\end{equation}
Choosing $w^j_1=-\Delta \dot{U}^{k\xi}$ in (\ref{L1}) and $w^j_2=\Delta \dot{D}^{k\xi}$ in (\ref{L2}), and using the fact that $g_i$ is nondecreasing, we find
%and using the fact that $g$ is nondecreasing, we find
\begin{equation*}\label{8} 
\begin{split}
&\frac{d}{dt}\int_\Omega \{|\nabla \dot{U}^{k\xi}(x,t) |^2+|\nabla \dot{D}^{k\xi}(x,t) |^2+|\nabla \Delta U^{k\xi}(x,t)|^2
+|\Delta D^{k\xi}(x,t)|^2\}\,dx\\&+2\frac{d}{dt}\int_\Omega a(x) \Delta D^{k\xi}(x,t)\Delta U^{k\xi}(x,t) \,dx \leq 0 \,\,\,\,\,  \forall   t\geq 0,
\end{split}
\end{equation*}
Integrating in $[0,t]$, to get
\begin{equation*}\label{812} 
\begin{split}
&\int_\Omega |\nabla \dot{U}^{k\xi}(t) |^2+|\nabla \dot{D}^{k\xi}(t) |^2\,dx+(1-c'\|a\|)\int_\Omega |\nabla \Delta U^{k\xi}(t)|^2+|\Delta D^{k\xi}(t)|^2\,dx \\&\leq 
C_2\int_\Omega \{|\nabla \dot{U}^{k\xi}(0) |^2+|\nabla \dot{D}^{k\xi}(0) |^2\int_\Omega |\nabla \Delta U^{k\xi}(0)|^2+|\Delta D^{k\xi}(0)|^2\}\,dx 
\end{split}
\end{equation*}
and $C_2$ is a positive constant depending only on $\|a\|$ and $c'$.\\Dividing by $\xi^2$, and letting $\xi \rightarrow 0$, we find 
\begin{equation*}\label{821} 
\begin{split}
&\int_\Omega \{|\nabla \ddot{u}^{k}_{1}(t) |^2+|\nabla \ddot{u}^{k}_{2}(t) |^2+|\nabla \Delta \dot{u}^{k}_1(t)|^2+|\Delta \dot{u}^{k}_2(t)|^2\}\,dx \\&\leq 
C'_2\int_\Omega \{|\nabla \ddot{u}^{k}_{1}(0) |^2+|\nabla \ddot{u}^{k}_{2}(0) |^2+|\nabla \Delta u^{1k}_1|^2+|\Delta u^{1k}_2|^2\}\,dx
\end{split}
\end{equation*}
We estimate $\|\nabla \ddot{u}^k_{i}(0)\|$. Choosing  $v=-\Delta\ddot{u}^k_{i}$ and $t=0$  in (\ref{PO}), we obtain that
$$
\|\nabla \ddot{u}^k_{1}(0)\|^2 =\int_\Omega  \nabla \ddot{u}^k_{1}( 0)\nabla(-\Delta^2 u^{0k}_1-a(x) u_2^{0k}+ g_1(\Delta u^{1k}_1) )\,dx.
$$
and
$$
\|\nabla \ddot{u}^k_{2}(0)\|^2 =\int_\Omega  \nabla \ddot{u}^k_{2}(0)\nabla(\Delta u^{0k}_2-a(x) u_1^{0k}+ g_2(\Delta u^{1k}_2) )\,dx.
$$
Using Cauchy-Schwarz inequality, we have
$$
\|\nabla \ddot{u}^k_{1}(0)\|\leq \Big(\int_\Omega  |\nabla(-\Delta^2 u^{0k}_1-a(x)u^{0k}_2+g_1(\Delta u^{1k}_1))|^2\,dx\Big)^{\frac{1}{2}}.
$$
and
$$
\|\nabla \ddot{u}^k_{2}(0)\|\leq \Big(\int_\Omega  |\nabla(\Delta u^{0k}_2-a(x)u^{0k}_1+g_2(\Delta u^{1k}_2))|^2\,dx\Big)^{\frac{1}{2}}.
$$
By (\ref{p8T}) and (\ref{p8F}) yields
\begin{equation}\label{O1}
(\ddot{u}_{1}^{k}(0),\ddot{u}_{2}^{k}(0))\,\,\,\,\, \mbox{are  bounded in}\,\,\,\,W\times W
\end{equation}
%We shall estimate $\|\nabla u^m_{tt}(x, 0)\|^2$. To this end, choose $v =-\Delta u^m_{tt}$ in (\ref{p6}) and set $t = 0$ to derive
%$$
%\|\nabla u^m_{tt}(x, 0)\|^2 =\int_\Omega  \nabla u^m_{tt}(x, 0)\nabla(-\Delta^2 u^m_0+ a(x)g(\Delta u^m_1) )\,dx
%$$
%from which, thanks to (\ref{p7})-(\ref{p8}), it easily follows that there exists a positive constant $C_0$ such that for every
%$m \geq 1$
%$$
%\|\nabla u^m_{tt}(x, 0)\| \leq C_0 
%$$
By (\ref{p8}), (\ref{p82}) and (\ref{O1}), we deduce that
\begin{equation*}\label{8s} 
\begin{split}
\int_\Omega \{|\nabla \ddot{u}^{k}_{1}(t) |^2+|\nabla \ddot{u}^{k}_{2}(t) |^2+|\nabla \Delta \dot{u}^{k}_1(t)|^2+|\Delta \dot{u}^{k}_2(t)|^2\}\,dx \leq C_3\,\,\,\,\,  \forall   t\geq 0,
\end{split}
\end{equation*}
where $C_3$ is a positive constant independent of $k\in \mathbb{N}$. Therefore, we
conclude that
\begin{equation}\label{S5}
 \dot{u}^k_1\,\,\,\,\, \hbox{is bounded in } \,\,\,\,\,\,\, L^\infty(0,T;H_\Delta^3(\Omega))
\end{equation}
\begin{equation}\label{S6}
 \dot{u}^k_2\,\,\,\,\, \hbox{is bounded in } \,\,\,\,\,\,\, L^\infty(0,T;H^2(\Omega))
\end{equation}
\begin{equation}\label{S7}
 \ddot{u}^k_{1}\,\,\,\,\, \hbox{is bounded in } \,\,\,\,\,\,\, L^\infty(0,T;H_0^1(\Omega))
 \end{equation}
 \begin{equation}\label{S8}
 \ddot{u}^k_{2}\,\,\,\,\, \hbox{is bounded in } \,\,\,\,\,\,\, L^\infty(0,T;H_0^1(\Omega))
 \end{equation}
 %{\it Passage to the limit:}
   Applying Dunford-Pettis and Banach-Alaoglu-Bourbaki theorems, we conclude from 
(\ref{S1})-(\ref{S44}) and  (\ref{S5})-(\ref{S8}) that there exists a subsequence $\{u_i^m\}$
of $\{u_i^k\}$ such that
%sequence  $ u^k $ with  a subsequence $u^m$ if needed, that
\begin{equation}\label{c1}
(u_1^m,u_2^m)\rightharpoonup (u_1,u_2),\hbox{   weak-star in   } L^\infty(0,T;\widetilde{V}),
\end{equation}
\begin{equation}\label{c2}
(\dot{u}^m_1,\dot{u}_2^m) \rightharpoonup (u'_1,u'_2)\hbox{   weak-star in   } L^\infty(0,T;V),
\end{equation}
\begin{equation}\label{c3}
(\ddot{u}^m_{1},\ddot{u}_2^m) \rightharpoonup (u''_1,u''_2)\hbox{   weak-star in    } L^\infty(0,T;W ),
\end{equation}
\begin{equation}\label{y4}
(\dot{u}^m_1,\dot{u}^m_2)\longrightarrow  (u'_1,u'_2),\hbox{   almost everywhere in    }\Omega\times [0,+\infty)
\end{equation}
\begin{equation}\label{b1}
g_i(\Delta \dot{u}^m_i)\rightharpoonup  \chi_i \hbox{    weak-star in  } L^{2}(\mathcal{A})
\end{equation}
%\begin{equation}\label{b2}
%g(\Delta \dot{u}_2^k)\rightharpoonup  \chi_2 \hbox{    weak-star in  } L^{(m+1)/m}([0,T]\times\Omega)
%\end{equation}
As $(u_1^m,u_2^m)$ is bounded in $L^\infty(0,T;\widetilde{V})$ (by (\ref{c1})) and the injection of $\widetilde{V}$ in $H$ is compact, we have
\begin{equation}\label{y7}
(u_1^m,u_2^m)\longrightarrow  (u_1,u_2),\hbox{   strong  in    }L^2(0,T;H).
\end{equation}
In the other hand, using  (\ref{c1}), (\ref{c3})  and (\ref{y7}), we have
\begin{equation}\label{F13FM}
\begin{split}
\int_{0}^{T}&\int_{\Omega}\Big(
\ddot{u}^m_1(x, t)+\Delta^2 u^m_1(x, t)-a(x) \Delta u_2^k(x,t) \Big)w \ dx\ dt
\longrightarrow\\& \int_{0}^{T}\int_{\Omega}\Big(u_1''(x, t)+\Delta^2 u_1(x, t)-a(x) \Delta u_2(x,t)\Big)w \ dx\ dt,
\end{split}
\end{equation}
and
\begin{equation}\label{F13FM}
\begin{split}
\int_{0}^{T}&\int_{\Omega}\Big(
\ddot{u}^m_2(x, t)-\Delta u^m_2(x, t)-a(x) \Delta u_1^m(x,t) \Big)w \ dx\ dt
\longrightarrow\\& \int_{0}^{T}\int_{\Omega}\Big(u_2''(x, t)-\Delta u_2(x, t)-a(x)\Delta  u_1(x,t)\Big)w \ dx\ dt,
\end{split}
\end{equation}
 for all  $w\in L^{2}(0,T;L^2(\Omega))$.\\
%%%%%%%%%%%%%%%%%%%%%%%%
It remains to show the convergence
$$
\int_{0}^{T}\int_{\Omega} g_i(\Delta \dot{u}^m_i) \ w\,dx\ dt\longrightarrow
\int_{0}^{T}\int_{\Omega} g_i(\Delta u'_i) \ w\,dx\ dt,
$$
when $m \rightarrow + \infty.$
%$$g(\Delta u'_i) \in L^1([0,T]\times\Omega).$$
%Indeed, since $g$ is continuous, we deduce from (\ref{y4})
%\begin{equation}
%\label{gk}
%g(\Delta \dot{u}^k_i)  \longrightarrow g(\Delta u'_i)\,\,\,\, \hbox{ almost everywhere in } \,\,\, [0,T]\times\Omega.
%\end{equation}
% It remains to show {\color{red}the convergence}
%  \begin{equation}\label{VV1}
 % \int_0^T \int_\Omega a(x)g(\Delta u_t^m)v\,dx\,dt  \longrightarrow \int_0^T \int_\Omega a(x)g(\Delta u_t)v\,dx\,dt.
 % \end{equation}
 \begin{lemma} 
For each $T>0$, $g_i(\Delta u'_i)\in L^1(\mathcal{A})$,
$ \|g_i(\Delta u'_i)\|_{L^1(\mathcal{A})}\leq K$, where $K$
is a constant independent of $t$ and $g_i(\Delta \dot{u}^k_i)\to g_i(\Delta u'_i)$ in $L^1(\mathcal{A})$.
\end{lemma}

\begin{proof}
We claim that
$$g(\Delta u') \in L^1(\mathcal{A}).$$ 
Indeed, since $g_i$ is continuous, we deduce from (\ref{y4})
\begin{equation}
\label{gk}
g_i(\Delta \dot{u}^k_i)  \longrightarrow g_i(\Delta u'_i)\,\,\,\, \hbox{ almost everywhere in } \,\,\, \mathcal{A}.
\end{equation}
$$
\Delta \dot{u}^k_i g_i(\Delta \dot{u}^k_i)  \longrightarrow \Delta u'_ig_i(\Delta u'_i)\,\,\,\, \hbox{ almost everywhere in } \,\,\, \mathcal{A}.
$$
Hence, by \eqref{S44} and Fatou's Lemma, we have
\begin{equation}\label{A1}
\int_0^T\int_{\Omega}\Delta u'_i(x,t)g_i(\Delta u'_i(x,t))\,dx\,dt\leq K_1, \,\quad 
\text{for } T>0
\end{equation}
Now, we can estimate
$\int_0^T\int_{\Omega}|g_i(\Delta u'_i(x,t))|\,dx\,dt$.
By Cauchy-Schwarz inequality, we have
\begin{align*}
\int_0^T\int_{\Omega}|g_i(\Delta u'_i(x,t))|\,dx\,dt
&\leq c|\mathcal{A}|^{1/2}
\Big(\int_0^T\int_{\Omega}| g_i(\Delta u'_i(x,t))|^2\,dx\,dt\Big)^{1/2}.
\end{align*}
Using (\ref{g1}) and (\ref{A1}), we obtain
\begin{equation*}
\begin{split}
\int_0^T\int_{\Omega}| g_i(\Delta u'_i(x,t))|^2\,dx\,dt &\leq \int_0^T\int_{|\Delta u'_i|> \varepsilon} \Delta u'_ig_i(\Delta u'_i)\,dx\,dt+
\int_0^T\int_{|\Delta u'_i|\leq \varepsilon} G^{-1}(\Delta u'_ig_i(\Delta u'_i))\,dx\,dt\\&
\leq c\int_{0}^T\int_\Omega \Delta u'_ig_i(\Delta u'_i)\,dx\,dt+cG^{-1}\Big(\int_{ \mathcal{A}} \Delta u'_ig_i(\Delta u'_i)\,dx\,dt\Big)
\\&\leq c\int_0^T\int_{\Omega} \Delta u'_ig_i(\Delta u'_i)\,dx\,dt+c'G^{*}(1)+c''\int_{\Omega} \Delta u'_ig(\Delta u'_i)\,dx\,dt \\&\leq cK_1+c'G^{*}(1),\,\,\,\,\,\,\,\, \mbox{for} \,\,\, T>0.
\end{split}
\end{equation*}
Then
$$
\int_0^T\int_{\mathcal{A}}|g_i(\Delta u'_i(x,t))|\,dx\,d \leq K, \,\,\,\,\,\,\,\, \mbox{for}\,\,\, T>0.
$$
Let $ E \subset \Omega \times [0,T]$ and set
$$
E_1=\Big\{ (x,t) \in E: |g_i(\Delta \dot{u}^m_i(x,t))| \leq \frac{1}{\sqrt{|E|}}\Big\} ,\quad
 E_2 =E\backslash E_1,
$$
where $|E|$ is the measure of $E$. If $M(r)=\inf \{ |s|: s\in \mathbb{R} \text{ and }
 |g_i(s)| \geq r\}$
$$
\int_E |g_i(\Delta \dot{u}^m_i)|\,dx\,dt
\leq c \sqrt{|E|}+ \Big(M\Big(\frac{1}{\sqrt{|E|}}\Big)\Big)^{-1}
\int_{E_2} |\Delta \dot{u}^m_ig_i(\Delta \dot{u}^m_i)|\,dx\,dt.
$$
By applying \eqref{S44} we deduce that
 $$
\sup_m \int_E g_i(\Delta \dot{u}^m_i)\ dx \ dt  \longrightarrow 0,\hbox{ when }
|E| \longrightarrow 0.
$$ 
From Vitali's convergence
theorem we deduce that
\begin{equation*}%\label{eqg}
g_i(\Delta \dot{u}^m_i) \to g_i(\Delta u'_i) \quad \text{in } L^1(\mathcal{A}).
\end{equation*}
This completes the proof.
 \end{proof}
Then (\ref{b1}) implies that 
\begin{equation*}
g_i(\Delta \dot{u}^m_i)\rightharpoonup g_i(\Delta u'_i) ,\hbox{    weak-star in  } L^{2}([0,T]\times\Omega).
\end{equation*}
We deduce, for all $v\in L^{2}([0,T]\times L^2(\Omega)$, that
$$
\int_{0}^{T}\int_{\Omega} g_i(\Delta \dot{u}^m_i)w\,dx\,dt\longrightarrow 
\int_{0}^{T}\int_{\Omega} g_i(\Delta u'_i)w\,dx\,dt.
$$
Finally we have shown that, for all $w\in L^{2}([0,T]\times L^2(\Omega))$:
$$
\int_{0}^{T}\int_{\Omega}\Big(u''_1(x, t)+\Delta^2 u_1(x, t)-a(x)\Delta u_2(x,t)-g_1(\Delta u'_1(x,t))\Big)w\,dx\,dt=0.
$$
and
$$
\int_{0}^{T}\int_{\Omega}\Big(u''_2(x, t)-\Delta u_2(x, t)-a(x)\Delta u_1(x,t)-g_2(\Delta u'_2(x,t))\Big)w\,dx\,dt=0.
$$
Therefore, $(u_1,u_2)$ are  a solutions for the problem (\ref{P}).
\section{Proof of  Theorem \ref{G} }
 From now on, we denote by c various positive constants which may be different on
different occurrences. Multiplying the first equation of  (\ref{P}) by $-\frac{\varphi{(E)}}{E}\Delta u_1$, we obtain
\begin{equation*}\label{FA7}
\begin{split}
0&=\int_S^T -\frac{\varphi{(E)}}{E} \int_\Omega \Delta  u_1( u''_1 +\Delta^2 u_1 -a(x)\Delta u_2+ g_1(\Delta u'_1))\,dx\,dt\\&=
-\Big[\frac{\varphi{(E)}}{E}\int_\Omega  u'_1\Delta u_1\,dx\Big]_S^T+\int_S^T \Big(\frac{\varphi{(E)}}{E}\Big)'\int_\Omega \Delta u_1 u'_1\,dx\,dt\\&
-2\int_S^T \frac{\varphi{(E)}}{E}\int_\Omega |\nabla u'_1|^2\,dx\,dt+ \int_S^T \frac{\varphi{(E)}}{E}\int_\Omega (|\nabla u'_1|^2+|\nabla \Delta u_1|^2)\,dx\,dt\\&
+\int_S^T \frac{\varphi{(E)}}{E}\int_\Omega a(x)\Delta u_1\Delta  u_2\,dx\,dt +\int_S^T \frac{\varphi{(E)}}{E}\int_\Omega \Delta u_1 . g_1(\Delta  u'_1)\,dx\,dt 
\end{split}
\end{equation*}
Similarly, we have
\begin{equation*}\label{FA7}
\begin{split}
0&=\int_S^T -\frac{\varphi{(E)}}{E} \int_\Omega \Delta  u_2( u''_2 +\Delta u_2 -a(x)\Delta u_1+ g_2(\Delta u'_2))\,dx\,dt\\&=
-\Big[\frac{\varphi{(E)}}{E}\int_\Omega  u'_2\Delta u_2\,dx\Big]_S^T+\int_S^T \Big(\frac{\varphi{(E)}}{E}\Big)'\int_\Omega \Delta u_2 u'_2\,dx\,dt\\&
-2\int_S^T \frac{\varphi{(E)}}{E}\int_\Omega |\nabla u'_2|^2\,dx\,dt+ \int_S^T \frac{\varphi{(E)}}{E}\int_\Omega (|\nabla u'_2|^2+| \Delta u_2|^2)\,dx\,dt\\&
+\int_S^T \frac{\varphi{(E)}}{E}\int_\Omega a(x)\Delta u_2\Delta  u_1\,dx\,dt +\int_S^T \frac{\varphi{(E)}}{E}\int_\Omega \Delta u_2 . g_2(\Delta  u'_2)\,dx\,dt 
\end{split}
\end{equation*}
%Similarly, if we multiply the second equation of  (\ref{p5})  by $\frac{\varphi{(E)}}{E} e^{-2\tau \rho}g_2(\Delta (z(x,\rho,t)))$
%we have
Taking their sum, we obtain
\begin{equation}\label{FA7S}
\begin{split}
 \int_S^T \varphi{(E)}\,dt &\leq  \Big[\frac{\varphi{(E)}}{E}\int_\Omega  u'_1\Delta u_1+u'_2\Delta u_2\,dx\Big]_S^T
 \\&-\int_S^T \Big(\frac{\varphi{(E)}}{E}\Big)'\int_\Omega \Delta u_1 u'_1+\Delta u_2 u'_2\,dx\,dt\\&
+2\int_S^T \frac{\varphi{(E)}}{E}\int_\Omega |\nabla u'_1|^2+|\nabla u'_2|^2\,dx\,dt
\\&-\int_S^T \frac{\varphi{(E)}}{E}\int_\Omega \Delta u_1 . g_1(\Delta u'_1)+\Delta u_2 . g_2(\Delta u'_2)\,dx\,dt
%\\&-\int_S^T \frac{\varphi{(E)}}{E}\int_\Omega \Delta u . g_2(\Delta z(x,1,t))\,dx\,dt\\&
 %-\Big[\frac{\varphi{(E)}}{E}\int_\Omega \int_0^1 \tau e^{-2\tau \rho}G(\Delta z)\,d\rho\,dx\Big]_S^T\\&+
%\tau\int_S^T\Big(\frac{\varphi{(E)}}{E}\Big)' \int_\Omega \int_0^1 e^{-2\tau \rho}G(\Delta z)\,dx\,d\rho\,dt\\&
%+2\tau \int_S^T\frac{\varphi{(E)}}{E} \int_\Omega \int_0^1 e^{-2\tau \rho}G(\Delta z)\,dx\,d\rho\,dt
\end{split}
\end{equation}
Since $E$ is non-increasing, we find that
$$
\Big[\frac{\varphi{(E)}}{E}\int_\Omega  u'_1\Delta u_1+u'_2\Delta u_2\,dx\Big]_S^T\leq c \varphi{(E(S))}
$$
$$
\Big|\int_S^T \Big(\frac{\varphi{(E)}}{E}\Big)'\int_\Omega \Delta u_1 u'_1+\Delta u_2 u'_2\,dx\,dt\Big| \leq c \varphi{(E(S))}
$$

Using these estimates, we conclude from (\ref{FA7S}) that
\begin{equation}\label{I1}
\begin{split}
\int_S^T \varphi{(E)}\,dt &\leq  C \varphi{(E(S))}+2\int_S^T \frac{\varphi{(E)}}{E}\int_\Omega |\nabla u'_1|^2+|\nabla u'_2|^2\,dx\,dt\\&
+\int_S^T \frac{\varphi{(E)}}{E}\int_\Omega |\Delta u_1| . |g_1(\Delta u'_1)|+ |\Delta u_2| . |g_2(\Delta u'_2)|\,dx\,dt \\&
 \end{split}
\end{equation}
Now, we estimate the terms of the right-hand side of (\ref{I1}) in order to apply the
results of Lemma \ref{0}.\\
As in Komornik \cite{KOM}, we consider the following partition of $\Omega$, 
$$
\Omega^+=\{x\in \Omega:|\Delta u'_i|> \epsilon\},\quad
\Omega^-=\{x\in \Omega:|\Delta u'_i|\leq \epsilon \}.
$$
We distinguish two cases:\\
{$\blacktriangleright$\bf{Case 1. $G$ is linear on  $[0,\epsilon]$}}.
By using Sobolev embedding and  Young's  inequality, we obtain
\begin{equation}\label{F1}
\begin{split}
&\int_S^T \frac{\varphi{(E)}}{E}\int_{\Omega^+} |\Delta u_1| . |g_1(\Delta  u'_1)|\,dx\,dt+\int_S^T \frac{\varphi{(E)}}{E}\int_{\Omega^+} |\nabla u'_1|^2\,dx\,dt \\&\leq \varepsilon \int_S^T \frac{\varphi{(E)}}{E}\int_{\Omega^+} |\Delta u_1|^2\,dx\,dt +C(\varepsilon )
\int_S^T \frac{\varphi{(E)}}{E}\int_{\Omega^+}  |g_1(\Delta  u'_1)|^2\,dx\,dt+c\int_S^T \frac{\varphi{(E)}}{E}\int_{\Omega^+} |\Delta u'_1|^2
\\&\leq \varepsilon c' \int_S^T \frac{\varphi{(E)}}{E}\int_\Omega |\nabla \Delta u_1|^2\,dx\,dt +(C(\varepsilon )c_2+\frac{c}{c_1})
\int_S^T \frac{\varphi{(E)}}{E}\int_\Omega  \Delta u'_1g_1(\Delta  u'_1)\,dx\,dt
\\&\leq \varepsilon C \int_S^T \varphi{(E)}\,dt 
+C_1(\varepsilon )\int_S^T \frac{\varphi{(E)}}{E}\int_\Omega  \Delta u'_1g_1(\Delta  u'_1)\,dx\,dt,
\end{split}
\end{equation}
Similarly, we have
\begin{equation}\label{F2}
\begin{split}
&\int_S^T \frac{\varphi{(E)}}{E}\int_{\Omega^+} |\Delta u_2| . |g_2(\Delta  u'_2)|\,dx\,dt+\int_S^T \frac{\varphi{(E)}}{E}\int_{\Omega^+} |\nabla u'_2|^2\,dx\,dt \\&\leq  \varepsilon C\int_S^T \varphi{(E)}\,dt 
+C_2(\varepsilon) \int_S^T \frac{\varphi{(E)}}{E}\int_\Omega  \Delta u'_2g_2(\Delta  u'_2)\,dx\,dt.
\end{split}
\end{equation}
Summing (\ref{F1}) and (\ref{F2}), and noting that $s \mapsto \frac{\varphi(s)}{s}$ is non-decreasing, we obtain
\begin{equation}\label{F11}
\begin{split}
\int_S^T \frac{\varphi{(E)}}{E}\int_{\Omega^+} &|\Delta u_1| . |g_1(\Delta  u'_1)|+|\Delta u_2| . |g_2(\Delta  u'_2)|\,dx\,dt\\&+\int_S^T \frac{\varphi{(E)}}{E}\int_{\Omega^+} |\nabla u'_1|^2+|\nabla u'_2|^2\,dx\,dt   \\&\leq
\varepsilon C \int_S^T \varphi{(E)}\,dt +C'(\varepsilon)\int_S^T \frac{\varphi{(E)}}{E} (-E'(t))\,dt
\\&\leq \varepsilon C \int_S^T \varphi{(E)}\,dt +C'(\varepsilon) \varphi{(E(S))}
\end{split}
\end{equation}
and
\begin{equation*}
\begin{split}
\int_S^T \frac{\varphi{(E)}}{E}\int_{\Omega^-} &|\Delta u_1| . |g(\Delta  u'_1)|+|\Delta u_2| . |g(\Delta  u'_2)|\,dx\,dt\\&+\int_S^T \frac{\varphi{(E)}}{E}\int_{\Omega^-} |\nabla u'_1|^2+|\nabla u'_2|^2\,dx\,dt   \\&\leq
\varepsilon C \int_S^T \varphi{(E)}\,dt +C'(\varepsilon)\int_S^T \frac{\varphi{(E)}}{E} (-E'(t))\,dt
\\&\leq \varepsilon C \int_S^T \varphi{(E)}\,dt +C'(\varepsilon) \varphi{(E(S))}.
\end{split}
\end{equation*}
Inserting these two inequalities into (\ref{I1}) and choosing $\varepsilon> 0$ small enough, we
deduce that
$$
\int_S^T \varphi{(E(t))}\,dt \leq c \varphi{(E(S))}
$$
Since, choosing $ \varphi{(s)}=s$, we deduce from (\ref{70K}) that
$$
E(t) \leq ce^{-\omega t}.
$$
{$\blacktriangleright$\bf{ Case 2. $G'(0)=0,\,\,\, G''>0$ on $]0,\epsilon]$}}\\
Using (\ref{g1}) and the fact that $s \mapsto \frac{\varphi(s)}{s}$ is non-decreasing, we obtain
\begin{equation*}\label{FA7}
\begin{split}
&\int_S^T \frac{\varphi{(E)}}{E}\int_{\Omega^+} |\nabla u'_1|^2+|\Delta u_1| . |g(\Delta u'_1)|\,dx\,dt \\&\leq \varepsilon C \int_S^T \varphi{(E)}\,dt+
\int_S^T \frac{\varphi{(E)}}{E}\int_{\Omega^+} \Delta u'_1g_1(\Delta u'_1)\,dx\,dt
%\\&\leq
%C \varphi{(E(S))}
 \end{split}
\end{equation*}
and
\begin{equation*}\label{FA7}
\begin{split}
&\int_S^T \frac{\varphi{(E)}}{E}\int_{\Omega^+} |\nabla u'_2|^2+|\Delta u_2| . |g(\Delta u'_2)|\,dx\,dt
\\&\leq \varepsilon C \int_S^T \varphi{(E)}\,dt+
\int_S^T \frac{\varphi{(E)}}{E}\int_{\Omega^+} \Delta u'_2g_2(\Delta u'_2)\,dx\,dt. \\&
\end{split}
\end{equation*}
Summing these two inequalities, we have
\begin{equation}\label{I5}
\begin{split}
&\int_S^T \frac{\varphi{(E)}}{E}\int_{\Omega^+} |\nabla u'_1|^2+|\Delta u_1| . |g_1(\Delta u'_1)|\,dx\,dt \\&
+\int_S^T \frac{\varphi{(E)}}{E}\int_{\Omega^+} |\nabla u'_2|^2+|\Delta u_2| . |g_2(\Delta u'_2)|\,dx\,dt
\\&\leq \varepsilon C \int_S^T \varphi{(E)}\,dt+C \varphi{(E(S))},
\end{split}
\end{equation}
and exploit Jensen's inequality and the concavity of $G^{-1}$ to obtain
\begin{equation}\label{I2}
\begin{split}
&\int_S^T \frac{\varphi{(E)}}{E}\int_{\Omega^-} |\Delta u_1| . |g_1(\Delta u'_1)|+|\nabla u'_1|^2\,dx\,dt 
%\\&\leq \varepsilon \int_S^T \frac{\varphi{(E)}}{E}\int_\Omega |\Delta u_1|^2\,dx\,dt +C(\varepsilon )
%\int_S^T \frac{\varphi{(E)}}{E}\int_\Omega  |g(\Delta  u'_1)|^2\,dx\,dt+c\int_S^T \frac{\varphi{(E)}}{E}\int_{\Omega^-} |\Delta u'_1|^2\,dx\,dt 
\\&\leq \varepsilon c' \int_S^T \frac{\varphi{(E)}}{E}\int_\Omega |\nabla \Delta u_1|^2\,dx\,dt +C(\varepsilon )
\int_S^T \frac{\varphi{(E)}}{E}\int_\Omega  (|\Delta u'_1|^2+ |g_1(\Delta u'_1)|^2)\,dx\,dt
\\&\leq \varepsilon c' \int_S^T \frac{\varphi{(E)}}{E}\int_\Omega |\nabla \Delta u_1|^2\,dx\,dt +C(\varepsilon )
\int_S^T \frac{\varphi{(E)}}{E}|\Omega | G^{-1}\Big( \frac{1}{|\Omega|}\int_\Omega \Delta u'_1g_1(\Delta u'_1)\,dx\Big)\,dt
\\&\leq \varepsilon C \int_S^T \varphi{(E)}\,dt +C(\varepsilon )
\int_S^T \frac{\varphi{(E)}}{E}|\Omega | G^{-1}\Big( \frac{1}{|\Omega|}\int_\Omega \Delta u'_1g_1(\Delta u'_1)\,dx\Big)\,dt
%C(\varepsilon) \varphi{(E(S))}
\end{split}
\end{equation}
Similarly, we have
\begin{equation}\label{I3}
\begin{split}
&\int_S^T \frac{\varphi{(E)}}{E}\int_{\Omega^-} |\Delta u_2| . |g_2(\Delta u'_2)|+|\nabla u'_2|^2\,dx\,dt 
%\leq \varepsilon \int_S^T \frac{\varphi{(E)}}{E}\int_\Omega |\Delta u_1|^2\,dx\,dt +C(\varepsilon )
%\int_S^T \frac{\varphi{(E)}}{E}\int_\Omega  |g(\Delta  u'_1)|^2\,dx\,dt
%\\&\leq \varepsilon c' \int_S^T \frac{\varphi{(E)}}{E}\int_\Omega |\nabla \Delta u_1|^2\,dx\,dt +C(\varepsilon )
%\int_S^T \frac{\varphi{(E)}}{E}\int_\Omega  (|\Delta u'_1|^2+ |g(\Delta u'_1)|^2)\,dx\,dt
\\&\leq \varepsilon  \int_S^T \frac{\varphi{(E)}}{E}\int_\Omega | \Delta u_2|^2\,dx\,dt +C(\varepsilon )
\int_S^T \frac{\varphi{(E)}}{E}|\Omega | G^{-1}\Big( \frac{1}{|\Omega|}\int_\Omega \Delta u'_2g_2(\Delta u'_2)\,dx\Big)\,dt
\\&\leq \varepsilon C \int_S^T \varphi{(E)}\,dt  +C(\varepsilon )
\int_S^T \frac{\varphi{(E)}}{E}|\Omega | G^{-1}\Big( \frac{1}{|\Omega|}\int_\Omega \Delta u'_2g_2(\Delta u'_2)\,dx\Big)\,dt
%C(\varepsilon) \varphi{(E(S))}
\end{split}
\end{equation}
Let $G^*$ denote the dual function of the convex function $G$ in the sense of
Young (see Arnold \cite[p. 64]{Ar}). Then $G^*$ is the Legendre transform of $G$,
which is given by (see Arnold \cite[p. 61-62]{Ar}) i.e.,
$$
G^*(s)= \sup_{t\in \mathbb{R}^+}(st-G(t)).
$$
Then $G^*$ is given by
$$
G^*(s)=s(G')^{-1}(s)-G[(G')^{-1}(s)],\,\,\,\,  \forall s\geq 0
$$
and satisfies the following inequality
\begin{equation}\label{A3}
st\leq G^*(s)+G(t)\,\,\,\,  \forall s,\, t\geq 0
\end{equation}
Choosing  $\varphi(s)=sG'( \epsilon s)$, we obtain 
\begin{equation}\label{A4}
G^{*} \Big(\frac{\varphi(s)}{s}\Big) =s \epsilon G'( \epsilon s) =  \epsilon s G'(\epsilon s)- G(\epsilon s) \leq \epsilon  \varphi(s) 
\end{equation}

%%%%%%%
%On the other hand, since $G$ is convex and increasing, $G^{-1}$ is concave and increasing.
 %The reversed Jensen's inequality for a concave function imply that
%\begin{equation}\label{A1}
%\begin{split}
%\int_S^T \frac{\varphi{(E)}}{E}\int_{\Omega^-} |\nabla u'_1|^2\,dx\,dt &\leq C\int_S^T \frac{\varphi{(E)}}{E}\int_\Omega (|\Delta u'_1|^2+|g(\Delta u'_1)|^2)\,dx\,dt
%\\&\leq C\int_S^T \frac{\varphi{(E)}}{E}\int_\Omega G^{-1}( \Delta u'_1g(\Delta u'_1))\,dx\,dt
%\\&\leq C\int_S^T \frac{\varphi{(E)}}{E}|\Omega | G^{-1}\Big( \frac{1}{|\Omega|}\int_\Omega \Delta u'_1g(\Delta u'_1)\Big)\,dx\,dt
%\\&\leq C_s\int_S^T \frac{\varphi{(E)}}{E} E'(t)\,dt
%\\&\leq C \varphi{(E(S))}
% \end{split}
%\end{equation}
%and
%\begin{equation}\label{A2}
%\begin{split}
%\int_S^T \frac{\varphi{(E)}}{E}\int_{\Omega^-} |\nabla u'_2|^2\,dx\,dt 
%\leq C\int_S^T \frac{\varphi{(E)}}{E}|\Omega | G^{-1}\Big( \frac{1}{|\Omega|}\int_\Omega \Delta u'_2g(\Delta u'_2)\Big)\,dx\,dt
%\\&\leq C_s\int_S^T \frac{\varphi{(E)}}{E} E'(t)\,dt
%\\&\leq C \varphi{(E(S))}
 %\end{split}
%\end{equation}
Making use of (\ref{A3}) and (\ref{A4}), we have
\begin{equation}\label{I4}
\begin{split}
& \int_S^T \frac{\varphi{(E)}}{E}|\Omega | G^{-1}\Big( \frac{1}{|\Omega|}\int_\Omega \Delta  u'_ig_i(\Delta u'_i)\Big)\,dx\,dt
  \\&\leq c \int_S^T G^{*}(\frac{\varphi(E)}{E}) \,dt +
c \int_S^T \int_\Omega\Delta u'_ig_i(\Delta u'_i)\Big)\,dx\,dt
\\&\leq c \int_S^T \varphi(E) \,dt  +c \int_S^T \int_\Omega\Delta u'_ig_i(\Delta u'_i)\Big)\,dx\,dt
\end{split}
\end{equation}
Summing (\ref{I2}) and (\ref{I3}) and using (\ref{I4}), we obtain
\begin{equation}\label{I6}
\begin{split}
&\int_S^T \frac{\varphi{(E)}}{E}\int_{\Omega^-} |\Delta u_1| . |g_1(\Delta u'_1)|+|\nabla u'_1|^2\,dx\,dt +\int_S^T \frac{\varphi{(E)}}{E}\int_{\Omega^-}|\Delta u_2| . |g_2(\Delta u'_2)| +|\nabla u'_2|^2 \,dx\,dt
\\&\leq \varepsilon C \int_S^T \varphi{(E)}\,dt  +C(\varepsilon ) E(S)
\end{split}
\end{equation}
Then, choosing $\varepsilon>0$ small enough  and substitution of (\ref{I5}) and (\ref{I6}) into (\ref{I1}) gives
%Then, choosing 0 > 0 small enough and using (4.2), we obtain in both cases
\begin{equation*}\label{FA7}
\begin{split}
\int_S^T \varphi(E(t))\,dt &\leq  c(E(S)+ \varphi(E(S)) )\\& \leq 
c\Big(1+ \frac{\varphi(E(S))}{E(S)}\Big) E(S)   \leq c E(S),\,\,\,  \forall S\geq 0
\end{split}
\end{equation*}
Using Lemma  \ref{0} in the particumar case where  $\Psi(s)=\omega \varphi(s)$ we deduce from
%Using Lemma 2.2 in the particular case where$\Psi(s)=\omega \varph(s)$ we deduce from
(\ref{77}) our estimate (\ref{I1}). The proof of Theorem \ref{G}  is now complete.

\begin{example}
Let $g_i$ be given by $ g_i(s)=s^p(-\ln s)^q$ where $p\geq 1$  and $q\in \mathbb{R}$ on $]0,\epsilon]$.
 Then $ g'_i(s)= s^{p-1}(-\ln s)^{q-1}(p(-\ln s)-q)$ which is an increasing
function in the right neighborhood of $0$  (if $q=0$ we can take $\epsilon =1$). The
function $G$ is defined in the neighborhood of $0$ by
$$
G(s)=cs^{\frac{p+1}{2}}(-\ln \sqrt{s})^q
$$
and we have
$$
G'(s)=cs^{\frac{p-1}{2}}(-\ln \sqrt{s})^{q-1}\Big(\frac{p+1}{2}(-\ln \sqrt{s})-\frac{q}{2}\Big),\,\,\,\,\, \mbox{when s is near } 0
$$
Thus
$$
\varphi(s)=cs^{\frac{p+1}{2}}(-\ln \sqrt{s})^{q-1}\Big(\frac{p+1}{2}(-\ln \sqrt{s})-\frac{q}{2}\Big),\,\,\,\,\, \mbox{when s is near } 0
$$
and 
\begin{equation*}\label{FA7}
\begin{split}
\psi(t)&= c\int_t^1 \frac{1}{s^{\frac{p+1}{2}}(-\ln \sqrt{s})^{q-1}\Big(\frac{p+1}{2}(-\ln \sqrt{s})-\frac{q}{2}\Big)}\,ds\\&
=c\int_1^{\frac{1}{\sqrt{t}}} (\ln z)^{q-1}\Big(\frac{p+1}{2}\ln z-\frac{q}{2}\Big)\,dz,\,\,\,\,\, \mbox{when t is near } 0
\end{split}
\end{equation*}
We obtain in the neighborhood of $0$
\begin{equation*}
\psi(t)=\left\{
\begin{aligned}
& c\frac{1}{t^{\frac{p-1}{2}}(-\ln t)^q}
 &\mbox{if}\,\,\, q>1& \\
 &c(-\ln t)^{q-1} &\mbox{if}\,\, p=1,\, q<1&
 \\
 &c(\ln (-\ln t)) &\mbox{if}\,\, p=1,\, q=1&\\
 %\\
%&u_i(x,0)=u_{i}^0(x), \, \,\, u'_{i}(x,0)=u_{i}^1(x),\   &x\in \Omega,\,\,\,\,\, i=1,2& \\
\end{aligned}
\right.
\end{equation*}
and then in the neighborhood of $+\infty$
\begin{equation*}
\psi^{-1}(t)=\left\{
\begin{aligned}
& ct^{-\frac{2}{p-1}}(\ln t)^{-\frac{2q}{p-1}}
 &\mbox{if}\,\,\, q>1& \\
 &ce^{-t^{\frac{1}{1-q}}} &\mbox{if}\,\, p=1,\, q<1&
 \\
 &ce^{-e^t}&\mbox{if}\,\, p=1,\, q=1&\\
 %\\
%&u_i(x,0)=u_{i}^0(x), \, \,\, u'_{i}(x,0)=u_{i}^1(x),\   &x\in \Omega,\,\,\,\,\, i=1,2& \\
\end{aligned}
\right.
\end{equation*}
Since $h(t)=t$ as $t$ tends to infinity, we obtain
\begin{equation*}
E(t)\leq \left\{
\begin{aligned}
& ct^{-\frac{2}{p-1}}(\ln t)^{-\frac{2q}{p-1}}
 &\mbox{if}\,\,\, q>1& \\
 &ce^{-t^{\frac{1}{1-q}}} &\mbox{if}\,\, p=1,\, q<1&
 \\
 &ce^{-e^t}&\mbox{if}\,\, p=1,\, q=1&\\
 %\\
%&u_i(x,0)=u_{i}^0(x), \, \,\, u'_{i}(x,0)=u_{i}^1(x),\   &x\in \Omega,\,\,\,\,\, i=1,2& \\
\end{aligned}
\right.
\end{equation*}
\end{example}


\bigskip



\begin{thebibliography}{10}

 %V.Komornik,  {\em Well-posedness and decay estimates for a Petrovsky system by a semigroup approach \/,} Acta Sci.Math (Szeged)
%{\bf 60}, (1995), 451-466.\\

%\bibitem{messaoudi}
\bibitem{ak} A. Ben Aissa, B. Gilbert and S. Nicaise, {\em Same decay rate of second order evolution equations with or without
delay}, Systems \& Control Letters,  {\bf 141}, (2020), 104700.
\bibitem{adam} R. A. Adams, {\em Sobolev spaces\/,} Academic press, Pure and Applied
Mathematics, vol. {\bf 65}, (1978).

\bibitem{Ala}F. Alabau-Boussouira,
   {\em Convexity and weighted intgral inequalities for energy decay rates of
nonlinear dissipative hyperbolic systems,}
   Appl. Math. Optim. {\bf 51}  (2005), 61-105. 

\bibitem{Ar}  
V. I. Arnold,  {\em Mathematical Methods of Classical Mechanics}, 2nd ed., Graduate Texts
in Math. {\bf 60}, Springer, New York, 1989.


\bibitem{Feng} M. Bahlil and F. Baowei, \emph{Global Existence and Energy Decay of Solutions to a Coupled Wave and Petrovsky System with Nonlinear Dissipations and Source Terms}, Mediterr. J. Math. 1-27, (2020).
\bibitem{AB} A. Ben Aissa, B. Gilbert and S. Nicaise , {\em Same decay rate of second order evolution equations with or without delay}, Systems \& Control Letters, 141 (2020), 104700.
\bibitem{Ga} A. Guesmia,
{\em  Energy Decay for a Damped Nonlinear Coupled System\/,}  Journal of Mathematical Analysis and Applications. {\bf 239}, 38-48 (1999).


\bibitem{KOM} V. Komornik,
{\em  Well-posedness and decay estimates for a Petrovsky system by a semigroup approach,} Acta Sci. Math. (Szeged) {\bf 60} (1995), 451-466.

\bibitem{KOM1}{V. Komornik}, \emph{Exact Controllability and Stabilization. The Multiplier Method}, Masson Wiley, Paris (1994).

\bibitem{Las1}
 I. Lasiecka, D. Toundykov,
 {\em Energy decay rates for the semilinear wave equation with nonlinear localized damping and source terms,}
  Nonlinear Anal, {\bf 64} (2006), 1757-1797.

\bibitem{Las} I. Lasiecka
{\em Stabilization of wave and plate-like equation with nonlinear dissipation on the
boundary,} J. Differential Equations, {\bf 79} (1989), 340-381.


\bibitem{Lio} J.L. Lions,
{\em Quelques M\'ethodes De R\'esolution Des Probl\'emes Aux Limites Nonlin\'eaires, Dunod Gautier-Villars}, Paris, 1969.

\bibitem{lions} J. L,Lions, {\em Contr\^olabilit\'e exacte, perturbations et stabilisation de syst\`emes distribu\'es. Tome 1}; RMA 8 (1988)
\bibitem{Liu} W. J. Liu, E. Zuazua,
{\em  Decay rates for dissipative wave equations,}
 Ricerche Mat, {\bf 48} (1999), 61-75.

\end{thebibliography}
\end{document}